\documentclass[11pt,leqno]{amsart}
\pdfoutput=1

\usepackage{amsmath,amssymb,amsthm}
\usepackage{mathtools}
\usepackage{caption}
\usepackage{subcaption}
\usepackage{siunitx}
\usepackage{stmaryrd}
\usepackage{textcomp}
\usepackage{empheq}
\usepackage{tikz}
\usepackage[backgroundcolor=white,textsize=tiny,linecolor=white,bordercolor=white]{todonotes}
\usepackage[utf8]{inputenc}

\usepackage{esint}
\usepackage{graphicx}
\usepackage{hyperref}
\usepackage{enumitem}

\newtheorem{definition}{Definition}[section]
\newtheorem{theorem}{Theorem}[section]
\newtheorem{lemma}{Lemma}[section]

\newtheorem*{maintheorem*}{Main Theorem}

\renewcommand{\i}{\ifmmode\mathit{\mathchar"7010 }\else\char"10 \fi}
\renewcommand{\j}{\ifmmode\mathit{\mathchar"7011 }\else\char"11 \fi}
\newcommand{\R}{\mathbb{R}}

\newcommand{\ddt}[1]{\frac{\mathrm{d}#1}{\mathrm{d}t}}

\newcommand{\pt}[1]{#1_t}
\newcommand{\ptt}[1]{#1_{tt}}
\newcommand{\px}[1]{#1_x }

\newcommand{\pxx}[1]{#1_{xx}}
\newcommand{\pxxx}[1]{#1_{xxx}}

\newcommand{\ptx}[1]{#1_{tx}}

\def\begi{\begin{itemize}}
\def\endi{\end{itemize}}
\def\bega{\begin{array}}
\def\enda{\end{array}}

\def\bel{\begin{equation}\label}
\def\eeq{\end{equation}}

\newenvironment{Assumptions}
{%

\begin{enumerate}}%
{\end{enumerate}}

{%

\begin{enumerate}}%
{\end{enumerate}}

\newcommand{\Dtp}{D_t^+}
\newcommand{\Dp}{D_+}
\newcommand{\Dm}{D_-}
\newcommand{\DD}{D^2}

\newcommand{\normInf}[1]{\norm{#1}_\infty}
\newcommand{\normLinf}[1]{\norm{#1}_{L^\infty(\Omega)}}

\newcommand{\normLinfR}[1]{\norm{#1}_{L^\infty(\R)}}

\newcommand{\normOne}[1]{\norm{#1}_1}
\newcommand{\normLone}[1]{\norm{#1}_{L^1(\Omega)}}

\newcommand{\normBV}[1]{\abs{#1}_{BV}}

\newcommand{\wlin}{w^{\Delta t}}
\newcommand{\wcon}{\overline{w^{\Delta t}}}
\newcommand{\wthcon}{\overline{w^{\theta,\Delta t}}}
\newcommand{\zthcon}{\overline{z^{\theta,\Delta t}}}

\newcommand{\zcon}{\overline{z^{\Delta t}}}
\newcommand{\Br}{{\mathcal{B}_r}}

\DeclarePairedDelimiter\abs{\lvert}{\rvert}
\DeclarePairedDelimiter\norm{\lVert}{\rVert}

\begin{document}
\title[A Convergent Scheme for the Variational Heat Equation]{A Convergent Finite Difference Scheme for the Variational Heat Equation}

\author{G. M. Coclite}
\address[Giuseppe Maria Coclite]{\newline
 Dipartimento di Matematica, \newline Universit\`a di Bari, \newline
 Via E.~Orabona 4,I--70125 Bari, Italy.}
\email[]{giuseppemaria.coclite@uniba.it}

\author[J. Ridder]{J. Ridder}
\address[Johanna Ridder]{\newline
    Department of Mathematics,
    \newline University of Oslo,
  \newline P.O.Box NO-1053, Blindern, Oslo-0316, Norway.} 
\email[]{johanrid@math.uio.no}
   
\author[N. H. Risebro]{N. H. Risebro}
\address[Nils Henrik Risebro]{\newline
    Department of Mathematics,
    \newline University of Oslo,
  \newline P.O.Box NO-1053, Blindern, Oslo-0316, Norway.} 
\email[]{nilshr@math.uio.no}

 \thanks{G. M. Coclite is member of the Gruppo Nazionale  per l'Analisi Matematica, 
  la Probabilit\`a e le  loro Applicazioni (GNAMPA) of the Istituto Nazionale di Alta Matematica (INdAM)}

\date{\today}

\maketitle

\begin{abstract}
The variational heat equation is a nonlinear, parabolic equation not in divergence form that arises as a model for the dynamics of the director field in a nematic liquid crystal. We present a finite difference scheme for a transformed, possibly degenerate version of this equation and prove that a subsequence of the numerical solutions converges to a weak solution. This result is supplemented by numerical examples that show that weak solutions are not unique and give some intuition about how to obtain the physically relevant solution.
\end{abstract}

\section{Introduction}
\label{sec:intro}
In this paper we investigate the Cauchy problem
\begin{equation}
\label{eq:VHE}
\begin{cases}
\pt u=c(u)\px{(c(u)\px u)},&\quad x\in\Omega,\,t>0\\
u(x,0)=u_0(x),&\quad x\in\Omega,
\end{cases}
\end{equation}
where $\Omega=\R$ or $\Omega=[0,1]$ with periodic boundary conditions. 
We assume that
\begin{Assumptions}
\item \label{ass:c} $c\in C^2(\R)$, $c\ge 0$, $\abs{\{\xi\,|\,c(\xi)=0\}}<\infty$, and, w.l.o.g., $c\le 1$,
\item \label{ass:init} $u_0\in W^{1,1}(\Omega)\cap W^{1,\infty}(\Omega)$, $u_{0,x}\in BV(\Omega)$.
\end{Assumptions}
We call~\eqref{eq:VHE} the ``variational heat equation'', because it can be derived from a variational principle, similar to the variational wave equation~\cite{HunterSaxton1991,Saxton1989,GlasseyHunterZheng1996,BressanZheng2006,ChenZheng2013}, see~\eqref{eq:varWave} below.

The variational heat equation arises in the context of the continuum theory for nematic liquid crystals as a model for the dynamics of the director field. Liquid crystals are materials in a state of matter between the solid and the liquid state. In the case of uniaxial nematic liquid crystals, this means that the elongated molecules can move freely like in a fluid, but tend to align along the same direction like in a crystal. On a macroscopic scale such a state can be described by two vector fields, the velocity field and the so-called director field, which are governed by the Ericksen-Leslie equations~\cite{Stewart2004-Book, Virga1994-Book, DeGennes1993-Book, Leslie1968, Leslie1979,Leslie1992,Ericksen1961}. The director field is a unit vector field that gives the average direction of the molecules at each point.

To arrive at equation~\eqref{eq:VHE}, we assume the simplified setting of a uniaxial nematic with no flow and a director field $\mathbf{n}$ that lies in the $x-y$ plane and varies only in $x$-direction. Then the director can be described by an angle $u$ as $\mathbf{n}=(\cos(u),\sin(u),0)$. The Oseen-Frank energy, which models the tendency of the director to align along the same direction everywhere, reduces to 
\begin{equation*}
  E=\int (c(u))^2 (\px u)^2\,dx\,,
\end{equation*} 
where 
\begin{align}\label{eq:cLC}
  c(u)=\sqrt{k_1 \cos(u)^2 + k_2\sin(u)^2}\,,
\end{align}
and $k_1$ and $k_2$ are the Oseen-Frank elastic constants corresponding to bend and splay deformations~\cite{Stewart2004-Book,DeGennes1993-Book,Virga1994-Book,Oseen1933,Frank1958}. In addition, the director is subject to the dissipation
\begin{equation*}
  D=\kappa \int (u_x)^2\,dx\,,
\end{equation*}
where $\kappa$ is the rotational viscosity coefficient. Together, a variational principle applied to the energy law
\begin{equation*}
  \ddt{} E=D\,,
\end{equation*}
and scaling $\kappa=1$ gives~\eqref{eq:VHE}, see~\cite{AursandRidder2015,AursandNapoliRidder2015}.

A similar model is the variational wave equation~\cite{HunterSaxton1991,Saxton1989},
\begin{equation}\label{eq:varWave}
  \ptt{u}=c(u)\px{(c(u)\px{u})}\,,
\end{equation}
which is derived in the same way from the Oseen-Frank energy, but neglecting dissipation and instead including inertia in the form of the kinetic energy
\begin{equation*}
  \int \sigma (\px{u})^2 \,dx\,,
\end{equation*}
where $\sigma$ is the rotational inertia of the director, scaled to $1$ in~\eqref{eq:varWave}. Typical values for the elastic constants $k_1$ and $k_2$ in~\eqref{eq:cLC} are of order $10^{-11}$--$10^{-12}$, the dissipation $\kappa$ is of order $10^{-1}$--$10^{-3}$, and the rotational inertia $\sigma$ is of order $10^{-13}$,~\cite{Stewart2004-Book,XuChangQingLei1987}. On small length scales, the term from the elastic energy and the dissipation can be of the same order. The inertia term however is usually dominated by the dissipation~\cite{AursandRidder2015}, therefore~\eqref{eq:VHE} is a more suitable model than~\eqref{eq:varWave} in most physical settings.

From a mathematical point of view, if $k_1$ and $k_2$ are strictly positive, i.e., $c>0$, equation~\eqref{eq:VHE} is a nonlinear, uniformly parabolic equation. While~\eqref{eq:varWave}, and also the combination of~\eqref{eq:VHE} and~\eqref{eq:varWave} where both $\pt{u}$ and $\ptt{u}$ are included, does not possess a unique classical solution~\cite{GlasseyHunterZheng1996,BressanZheng2006,ChenZheng2013}, standard theory of nonlinear parabolic equations guarantees well-posedness of~\eqref{eq:VHE}, see~\cite{Ladyzhenskaya-book1968}.

We are therefore interested in the degenerate case of~\eqref{eq:VHE} where $c$ is allowed to vanish at some points, i.e., if $c$ is given by~\eqref{eq:cLC}, in the case that $k_1=0$ or $k_2=0$. Solutions of degenerate parabolic equations are not necessarily smooth or unique, therefore new concepts of solutions, e.g., weak solutions, entropy solutions, or viscosity solutions are required.
In the case of~\eqref{eq:VHE}, a formal calculation shows that there is no maximum principle for $\px{u}$, but for $c(u)\px{u}$ (see Section~\ref{sec:bounds}). 
At points where $c(u)$ vanishes, this allows for gradient blow-up.

The goal of this paper is to design a convergent numerical scheme for~\eqref{eq:VHE}. The form of the right-hand side and the resulting lack of a gradient bound suggests that one should transform~\eqref{eq:VHE} first. 

One possibility to do this is to define
\begin{equation}\label{eq:vTrafo}
  v=k_v(u)=\int^u \frac{1}{c(\xi)} \,d\xi\,.
\end{equation}
Then~\eqref{eq:VHE} becomes
\begin{equation}\label{eq:v}
  \pt{v}=\px{(c^2(\bar{k}_v(v)) \px{v})}\,,
\end{equation}
where $\bar{k}_v$ is the inverse of $k_v(u)$. For this equation it is straightforward to obtain an $L^2$ bound and one can also show uniqueness of weak solutions. If we assume $c>0$, a simple finite difference scheme based on central differences and averages in space can be shown to converge to a weak solution using Aubin-Lions lemma, see also~\cite{Ladyzhenskaya-book1985, Samarskii-book2001} for examples in a similar setting. If however $c=0$ for some $u$, then~\eqref{eq:vTrafo} is not necessarily finite and a bound on $\px{v}$  does not follow directly from the $L^2$ bound. 

An alternative transformation of~\eqref{eq:VHE} is
\begin{equation}\label{eq:wTrafo}
  w=k_w(u)=\int^u c(\xi)\,d\xi\,,
\end{equation} 
so $w$ satisfies
\begin{equation}\label{eq:w}
  \pt{w}=c^2(\bar{k}_w(w))\pxx{w}\,.
\end{equation}
The transformation $k_w$ and its inverse $\bar{k}_w$ are well-defined for any $c\ge 0$ if $c$ vanishes only on single points. It is also possible to show \emph{a priori} bounds for both $w$ and $\px{w}$ in $L^\infty$ and $BV$ (functions of bounded total variation), see  Section~\ref{sec:bounds}. However,~\eqref{eq:w} does not guarantee uniqueness of solutions. Indeed, Ughi et al.~\cite{Ughi1984,DalPassoLuckhaus1987,BertschDalPassoUghi1992} showed that for the special case where $c^2(\bar{k}_w(w))=w$, weak solutions of~\eqref{eq:w} (defined in a standard way, see Section~\ref{sec:definitions}) are not unique. To choose the physically relevant solution, they define ``viscosity solutions'' which are obtained by taking the limit of classical solutions of the equation with $c>0$ or suitable initial data. In the setting of~\eqref{eq:cLC}, these viscosity solutions correspond to sending $k_1$ or $k_2$ to $0$ or choosing the solution that corresponds to a solution of~\eqref{eq:v}.
Ughi et al.'s concept of viscosity solutions is not generally the same as Lions' theory of viscosity solutions for degenerate parabolic equations~\cite{CrandallIshiiLions1992,BertschDalPassoUghi1992}. The uniqueness theory of the latter is not applicable here, because the right-hand side of~\eqref{eq:w}, or~\eqref{eq:VHE}, is not proper.

The scheme that we will present in this paper discretizes~\eqref{eq:w}. Based on discrete versions of the $L^\infty$ and $BV$ bounds on $w$ and $\px{w}$, we use Kolmogorov's compactness theorem to show that the numerical approximations for both $w$ and $\px{w}$ converge strongly in $L^1(\Omega)$. The strong convergence of the derivative is important, because the weak formulation of~\eqref{eq:w} includes nonlinear terms in~$\px{w}$. Passing to the limit in the definition of the scheme, we prove that a subsequence of the numerical solutions converges to a weak solution as $\Delta x,\Delta t\rightarrow 0$.

Our numerical experiments confirm the nonunqiueness properties discussed above. If $k_1=0$ in~\eqref{eq:cLC} and the grid is chosen such that $c(u_0(x))$ is positive at every grid point, then the numerical solutions converge to Ughi et al.'s viscosity solution. This solution is the same as the one obtained by a method based on~\eqref{eq:v} and as the limit $k_1\rightarrow 0$ of solutions of the $w$-based scheme for any set of grid points. If however one of the grid points coincides with a zero of $c(u_0(x))$, we  get another solution which corresponds to a classical solution of~\eqref{eq:w}, ``glued together'' at the zeros of $c(u_0(x))$ with Dirichlet boundary conditions. Interpreted as solutions of~\eqref{eq:VHE}, this type of solutions shows clearly that the gradient is unbounded.

The rest of this paper is structured as follows: In Section~\ref{sec:definitions} we will define the scheme for~\eqref{eq:w}, introduce the notion of weak solutions, and state our convergence result. Section~\ref{sec:bounds} contains discrete \emph{a priori} bounds, which are based on Harten's lemma and motivated by formal calculations in the continuous case. Time continuity is shown in Section~\ref{sec:time} and the convergence proof is carried out in Section~\ref{sec:convergence}. In Section~\ref{sec:numerics} we present a series of numerical experiments that confirm the convergence result and highlight the nonuniqueness properties of~\eqref{eq:w}. 

\section{A numerical scheme for $w$ and the main result}\label{sec:definitions}
To be precise, let us restate~\eqref{eq:w} in the form that will be the basis of our scheme. 
Assume that
\begin{Assumptions}\setcounter{enumi}{2}
\item  \label{ass:B} $B\in C^2(\R)$ and $0\le B\le 1$,
\item \label{ass:winit} $w_0\in W^{1,1}(\Omega)\cap W^{1,\infty}(\Omega)$, $w_{0,x}\in BV(\Omega)$.
\end{Assumptions}
Then we want to solve
\begin{equation}\label{eq:weq}
\begin{cases}
  \pt{w}=B(w) \pxx{w}, &\quad t>0, x\in \Omega,\\
w(x,0)=w_0(x),&\quad x \in \Omega,
\end{cases}
\end{equation}
on $\Omega=\R$ or $[0,1]$ with periodic boundary conditions.

Equation~\eqref{eq:VHE} can be transformed to~\eqref{eq:weq} by defining $w$ as in~\eqref{eq:wTrafo}. If $u_0$ satisfies~\ref{ass:init}, then $w_0$ will satisfy~\ref{ass:winit}, but not vice versa. Similarly,~\ref{ass:B} follows from~\ref{ass:c}. As an example, if we choose $c$ according to~\eqref{eq:cLC} with $k_1=0$ and $k_2=1$, then  $k_w(u)=\abs{\sin(u)}$ and $B(w)=1-w^2$, see also Section~\ref{sec:numerics}. 

To define the scheme, let $\Omega$ be discretized by the equidistant grid points $x_j=j\Delta x$, $j=0,\dots, N$, and let $t^n= n\Delta t$ denote the time steps. If $\Omega=[0,1]$, we set periodic boundary conditions. We will implicitly assume that all functions are periodically extended outside of the domain, so that no boundary terms occur. 

A straightforward discretization of~\eqref{eq:weq} is 
\begin{equation}\label{eq:wScheme}
  \Dtp w^n_j=B(w^{n+\theta}_j)\DD w^{n+\theta}_j\,,
\end{equation}
where we used the difference quotients
\begin{align*}
&  \Dp a_j = \frac{1}{\Delta x} (a_{j+1}-a_j)\,, \quad &&\Dm a_j=\frac{1}{\Delta x} (a_j-a_{j-1})\,,\\
&  \Dtp a^n=\frac{1}{\Delta t} (a^{n+1}-a^n),\quad &&\DD a_j=\Dp\Dm a_j\,, 
\end{align*}
and the convex combination
\begin{equation*}
  w^{n+\theta}_j=\theta w^{n+1}_j + (1-\theta) w^n_j\,,\quad\text{where $\theta\in [0,1]$.}
\end{equation*}
For $\theta=0$, the scheme is explicit, for $\theta=1$, it is fully implicit, and for $\theta=\frac{1}{2} $ we have the Crank-Nicholson time discretization. In the fully implicit case of $\theta=1$, the scheme is unconditionally stable. Otherwise, we require that the time step $\Delta t$ and grid size $\Delta x$ satisfy the CFL condition
\begin{equation}\label{eq:CFL}
  \lambda=\frac{\Delta t}{(\Delta x)^2} < \frac{1}{2(1-\theta)}. 
\end{equation}
For the discrete derivatives $z^n_j=\Dp w^n_j$ and $y^n_j=\Dm z^n_j$, the scheme defined by~\eqref{eq:wScheme} becomes
\begin{align}
  \Dtp z^n_j&=\Dp (B(w^{n+\theta}_j) \Dm z^{n+\theta}_j)\,,\label{eq:zScheme}\\
  \Dtp y^n_j&= \DD (B(w^{n+\theta}_j) y^{n+\theta}_j)\,.\label{eq:yScheme}
\end{align}
We will use these forms below to get \emph{a priori} bounds on $w^n_j$.

For given initial data $w_0\in W^{2,1}\cap W^{1,\infty}$, define the discrete initial data
\begin{equation}\label{eq:ID}
  w^0_j=\frac{1}{\Delta x} \int_{x_{j-\frac{1}{2}}}^{x_{j+\frac{1}{2}}} w_0(x) \,dx\,.
\end{equation}
To get from the discrete approximations $w^n_j$ back to continuous functions, we use the piecewise linear and piecewise constant interpolations
\begin{align}
  \wlin(x,t)&=\frac{x_{j+1}-x}{\Delta x} w^n_j + \frac{x-x_j}{\Delta x} w^n_{j+1}\,,\label{eq:wlin}\\ 
  &\qquad\qquad\qquad\text{for $x\in [x_j,x_{j+1})$, $t\in [t^n,t^{n+1})$,}\notag\\
\wcon(x,t)&=w^n_j\,, \\ 
&\qquad\qquad\qquad\text{for $x\in [x_{j-\frac{1}{2}},x_{j+\frac{1}{2}})$, $t\in [t^n, t^{n+1})$},\notag\\
\zcon(x,t)&=\wlin_x(x,t)=\Dp w^n_j=z^n_j\,,\\ 
&\qquad\qquad\qquad\text{for $x\in [x_j,x_{j+1})$, $t\in [t^n,t^{n+1})$.}\notag
\end{align}

Our main result is the convergence of the numerical scheme. 
Since $B(w)$ is allowed to vanish, equation~\eqref{eq:weq} is a
degenerate parabolic equation and solutions are not necessarily
smooth. In particular, the derivative of $w$ may not be defined at every point. We will therefore prove convergence to weak solutions of~\eqref{eq:weq}.
\begin{definition}[Weak solutions of~\eqref{eq:weq}]\label{def:weakSolution}
A function $w\in\, L^\infty(0,\infty;H^1(\Omega))\times L^\infty(\Omega\times(0,\infty))$ is a weak solution of~\eqref{eq:weq} if it satisfies
\begin{equation}\label{eq:weakform}
  \int_0^\infty\int_\Omega w \pt{\phi} - B(w) \px{w}\px{\phi} - B'(w) (\px{w})^2 \phi \, dx dt + \int_\Omega w_0(x)\phi(x,0)\,dx=0\,, 
\end{equation}  
for all $\phi \in C^\infty_c(\Omega\times [0,\infty))$.
\end{definition}
The convergence result, which we will prove in Section~\ref{sec:convergence}, reads as follows.
\begin{theorem}\label{thm:convergence}
A subsequence of the interpolations $w_{\Delta t}$ of the solutions of the scheme defined by~\eqref{eq:wScheme} converges in $C([0,\infty),W^{1,1}(\Omega))$ to a weak solution of~\eqref{eq:weq} as defined in Definition~\ref{def:weakSolution}.  
\end{theorem}
Note that only a subsequence of $w_{\Delta t}$ converges, because weak solutions of~\eqref{eq:weq} are not unique. We will comment more on this in Section~\ref{sec:numerics}.

For the a priori bounds in the next section, we will use the discrete norms
\begin{align*}
  \normInf{a^n}=\sup_j\, \abs{a^n_j}\,,\quad \normOne{a^n}=\Delta x \sum_j \abs{a^n_j}\,,\quad \normBV{a^n}=\sum_j \abs{a^n_j-a^n_{j-1}}\,,
\end{align*}
\section{A priori bounds}\label{sec:bounds}
In the following, we will show discrete maximum principles and $BV$ bounds for $w^n_j$ and $z^n_j=\Dp w^n_j$. Here, note that
the original equation~\eqref{eq:VHE} only possesses a maximum principle for $u$, but not for $\px{u}$, since in
\begin{equation*}
  \ptx{u} =(c(u))^2 \pxxx{u}+ 4 c(u) c'(u) \px{u} \pxx{u} + \frac{1}{2} (c^2(u))'' (\px{u})^3\,,
\end{equation*}
the third term can lead to growth of local maxima in $\px{u}$. Our numerical examples in Section~\ref{sec:numerics} confirm this.
One advantage of the transformation to $w$ is that for equation~\eqref{eq:weq} both $w$ and $z=\px{w}$ are bounded in $L^\infty$.

The $BV$ bound for $z$ will be important in the convergence proof, because strong convergence for both $w$ and its first derivative is needed to pass to the limit in the third term of the weak formulation~\eqref{eq:weakform}. Before turning to the discrete setting, 
let us show formally how $L^1$ bounds for $z$ and $y=\px{z}$ (i.e., $BV$ bounds for $w$ and $z$) can be obtained in the continuous case.

For $z$, multiply
\begin{equation*}
  \pt{z}=\px{(B \px{z})}
\end{equation*}
by $\eta'(z)$, where $\eta$ is some convex smooth function, and integrate in space to get
\begin{align*}
  \ddt{} \int_\Omega \eta(z)\,dx = - \int_\Omega B(w) (\px{z})^2 \eta''(z)\,dx \le 0. 
\end{align*}
Letting $\eta\rightarrow \abs{\cdot}$, we get an $L^1$ bound for $z$.

For $y$, the formal continuous equivalent of equation~\eqref{eq:yScheme} is
\begin{equation}\label{eq:yeq}
  \pt{y}=\pxx{(B(w)y)}.
\end{equation}
Again, let $\eta\in C^2(\R)$ be convex and multiply~\eqref{eq:yeq} by $\eta'(y)$. Then, 
\begin{align*}
  \pt{\eta(y)}&=(B \pxx{y}+ 2 \px{B} \px{y} + \pxx{B} y) \eta'(y)\\
&\leq   (\px{y})^2 B\eta''(y)+B \pxx{y} \eta'(y) + 2 \px{B} \px{\eta(y)} + \pxx{B} y\, \eta'(y)\\
&= B \pxx{\eta(y)} +  2 \px{B} \px{\eta(y)} + \pxx{B} y\, \eta'(y)\\
&= \px{(B \px{\eta(y)})} + \px{B}\px{\eta(y)} + \pxx{B} y\, \eta'(y)\\
&=\px{(B \px{\eta(y)})} + \px{(\px{B}\eta(y))} - \pxx{B}\eta(y) + \pxx{B} y\, \eta'(y)\\
&= \pxx{(B \eta(y))} + \pxx{B} (\eta'(y) y -\eta)\,. 
\end{align*}
Integrating over $\Omega$ and taking $\eta(y)=\abs{y}_\epsilon$ such that it converges to $\abs{y}$ as $\epsilon\rightarrow 0$, we get 
\begin{align*}
  \ddt{} \int_\Omega \abs{y}\le 0\,.
\end{align*}  

In the discrete case, we will base our proofs on an extended version of Harten's Lemma~\cite[p.~118]{HoldenRisebro-book}.
\begin{lemma}
Let $v_j$ be given by
\begin{equation}\label{eq:Harten1}
  v_j=u_j - A_{j-1/2}\Delta_- u_j + B_{j+1/2}\Delta_+ u_j - C_{j-1/2}\Delta_- v_j + D_{j+1/2} \Delta_+ v_j\,,
\end{equation}  
where $\Delta_{\pm}u_j=\pm (u_{j\pm 1}-u_j)$.
\begin{enumerate}[label=\textup{(\roman*)}]
  \item \label{cond:Harten1} If $A_{j+1/2}$, $B_{j+1/2}$, $C_{j+1/2}$, and $D_{j+1/2}$ are nonnegative for all $j$, and $A_{j+1/2}+B_{j+1/2}\le 1$ for all $j$, then
\begin{equation*}
  \normBV{v}\le \normBV{u}\,.
\end{equation*}
\item \label{cond:Harten2} If $A_{j+1/2}$, $B_{j+1/2}$, $C_{j+1/2}$, and $D_{j+1/2}$ are nonnegative for all $j$, and $A_{j-1/2}+B_{j+1/2}\le 1$ for all $j$, then
\begin{equation*}
  \min_i u_i\le v_j\le \max_i u_i
\end{equation*}
\end{enumerate}
\end{lemma}
\begin{proof}
From~\eqref{eq:Harten1}, we get
\begin{align*}
  (1+C_{j+1/2}+D_{j+1/2})\Delta_+ v_j&=(1-A_{j+1/2}-B_{j+1/2})\Delta_+u_j\\&\qquad{}+ A_{j-1/2}\Delta_- u_j + B_{j+3/2}\Delta_+u_{j+1}\\&\qquad {}+ C_{j-1/2}\Delta_-v_j + D_{j+3/2}\Delta_+v_{j+1}\,.
\end{align*}
Hence, under the assumptions of~\ref{cond:Harten1},
\begin{align*}
  \sum_j (1+C_{j+1/2}+D_{j+1/2}) \abs{\Delta_+ v_j}&\le \sum_j  (1-A_{j+1/2}-B_{j+1/2})\abs{\Delta_+u_j}  \\&\qquad {}+ \sum_j A_{j-1/2}\abs{\Delta_- u_j}+ B_{j+3/2}\abs{\Delta_+u_{j+1}} \\&\qquad{}+\sum_j C_{j-1/2}\abs{\Delta_-v_j} + D_{j+3/2}\abs{\Delta_+v_{j+1}}\\
&= \sum_j \abs{\Delta_+ u_j} + (C_{j+1/2}+D_{j+1/2})\abs{\Delta_+ v_j}\,,
\end{align*}
from which the $BV$ bound follows.

For the maximum principle, we can write~\eqref{eq:Harten1} as
\begin{align*}
    (1+C_{j-1/2}+D_{j+1/2}) v_j&=(1-A_{j-1/2}-B_{j+1/2}) u_j + A_{j-1/2}u_{j-1} + B_{j+1/2}u_{j+1}\\&\qquad {}+ C_{j-1/2}v_{j-1} + D_{j+1/2}v_{j+1}\,.
\end{align*}
Thus, if the assumptions of~\ref{cond:Harten2} hold, $v_{j'}=\max_i v_i$ satisfies
\begin{align*}
    (1+C_{j'-1/2}+D_{j'+1/2}) v_{j'}&\le (1-A_{j'-1/2}-B_{j'+1/2}) \max_i u_i\\&\qquad{}+ A_{j'-1/2}\max_i u_i + B_{j'+1/2}\max_i u_i\\&\qquad {}+ C_{j'-1/2} v_{j'} + D_{j'+1/2} v_{j'}\,,  
\end{align*}
and hence, $\max_i v_i =v_{j'}\le \max_i u_i$. Similarly, $\min_i v_i \ge \min_i u_i$, which concludes the proof.
\end{proof}

The $L^\infty$ and $BV$ bound for $w^n_j$ and $z^n_j$ follow directly from the above lemma.
\begin{lemma}\label{lem:bounds}
  Let $w^n_j$ be the solution of~\eqref{eq:wScheme} and $z^n_j=\Dp w^n_j$. Then
\begin{align*}
  \min_i w^0_i\le w^n_j\le \max_i w^0_i, \quad&\normBV{w^n}\le \normBV{w^0}, \\
\min_i z^0_i\le z^n_j\le \max_i z^0_i, \quad&\normBV{z^n}\le \normBV{z^0}.
\end{align*}
\end{lemma}
\begin{proof}
  Rewriting~\eqref{eq:wScheme}, we get
\begin{align*}
  w^{n+1}_j&=w^n_j+(1-\theta)\Delta t B(w^{n+\theta}_j)\DD w^n_j+\theta \Delta t B(w^{n+\theta}_j)\DD w^{n+1}_j.
\end{align*}
To apply Harten's lemma, set $v_j=w^{n+1}_j$, $u_j=w^n_j$, and
\begin{align*}
  A_{j-1/2}&=(1-\theta)\lambda B(w^{n+\theta}_j),&C_{j-1/2}&=\theta\lambda B(w^{n+\theta}_j),\\ 
  B_{j+1/2}&=(1-\theta)\lambda B(w^{n+\theta}_j),&D_{j+1/2}&=\theta\lambda B(w^{n+\theta}_j),
\end{align*}
where $\lambda=\Delta t/(\Delta x)^2$. Because $\lambda$ satisfies the CFL condition~\eqref{eq:CFL} and  $\theta$ and $B(w)$ take values in $[0,1]$, the assumptions of Harten's lemma hold and we get the maximum and $BV$ bound for $w^n_j$.

For $z$, write~\eqref{eq:zScheme} as
\begin{align*}
    z^{n+1}_j&=z^n_j + (1-\theta)\Delta t \Dp (B(w^{n+\theta}_j) \Dm z^n_j) +\theta \Delta t \Dp (B(w^{n+\theta}_j) \Dm z^{n+1}_j).
\end{align*} 
Set $v_j=z^{n+1}_j$, $u_j=z^n_j$, and
\begin{align*}
  A_{j-1/2}&=(1-\theta)\lambda B(w^{n+\theta}_j),&C_{j-1/2}&=\theta\lambda B(w^{n+\theta}_j),\\ 
  B_{j+1/2}&=(1-\theta)\lambda B(w^{n+\theta}_{j+1}),& D_{j+1/2}&=\theta\lambda B(w^{n+\theta}_{j+1}),
\end{align*}
in Harten's lemma. Again, due to the CFL condition and the bounds on $B$, the conditions are satisfied and the claim follows.
\end{proof}
\section{Continuity in time}\label{sec:time}
In order to show compactness, we will need continuity in time of both $\wlin$ and $\zcon$. For $\wlin$ this follows directly from the definition of the scheme and the $BV$ bound for $z$ above. 
\begin{lemma}\label{lem:wtime}
Let $\wlin$ be the interpolation~\eqref{eq:wlin} of the solutions $w^n_j$ of~\eqref{eq:wScheme}. Then, for any $t,t+\tau\ge 0$, 
\begin{equation*}
  \int_\Omega \abs{\wlin(x,t+\tau)-\wlin(x,t)} \,dx\le (\tau + \mathcal{O}(\Delta t))\normBV{z^0}+ \mathcal{O}(\Delta x)\normBV{w^0}.
\end{equation*}
\end{lemma}
\begin{proof}
Using the piecewise constant interpolation~$\wcon$, we get
\begin{align}
  \int_\Omega \abs{\wlin(x,t+\tau)-\wlin(x,t)}\,dx &\le \int_\Omega \abs{\wlin(x,t+\tau)-\wcon(x,t+\tau)} \nonumber \\&\hphantom{\le\int_\Omega}\qquad{}+ \abs{\wlin(x,t)-\wcon(x,t)} \label{eq:time1} \\&\hphantom{\le\int_\Omega}\qquad{}+ \abs{\wcon(x,t+\tau)-\wcon(x,t)}\,dx\,.\nonumber
\end{align}
Regarding the first two terms on the right-hand side, note that for $t\in[t^n,t^{n+1})$,
\begin{equation}\label{eq:wconDiff}
  \begin{split}
  \int_\Omega \abs{\wlin(x,t)-\wcon(x,t)}\,dx&=\sum_j\int_{x_{j-\frac{1}{2}}}^{x_j} \abs{(x_j-x)\Dm w^n_j } dx\\&\hphantom{\sum_j}\qquad{}+ \int_{x_j}^{x_{j+\frac{1}{2}}} \abs{(x-x_j)\Dp w^n_j } dx\\
&= \frac{(\Delta x)^2}{4} \sum_j \abs{\Dp w^n_j}\\
                                   &=\frac{\Delta x}{4} \normBV{ w^n} \le \frac{\Delta x}{4} \normBV{w^0}\,,      
  \end{split}
\end{equation}
where the last inequality is due to Lemma~\ref{lem:bounds}.
For the last term in~\eqref{eq:time1}, let $m$, $n$ be such that $t+\tau\in [t^n,t^{n+1})$ and $t\in[t^m,t^{m+1})$. Using the $BV$ bound on $z$ from Lemma~\ref{lem:bounds}, we get
\begin{align*}
  \int_\Omega \abs{\wcon(x,t+\tau)-\wcon(x,t)}\,dx&=\sum_j \int_{x_{j-\frac{1}{2}}}^{x_{j+\frac{1}{2}}}\abs{w^n_j - w^m_j}\\
&=\Delta x \sum_j \sum_{k=m}^{n-1} \Delta t \abs{\Dtp w^k_j}\\
&=\Delta x \Delta t \sum_j \sum_{k=m}^{n-1} \abs{B(w^{k+\theta}_j) \Dm z^{k+\theta}_j}\\
&\le \Delta t (n-m) \normBV{z^0}= (\tau + \mathcal{O}(\Delta t)) \normBV{z^0}\,,
\end{align*}
and the claim follows.
\end{proof}
For $\zcon$, we will use a version of Kru\v{z}kov's interpolation lemma~\cite[p.~208, Lemma~4.11]{HoldenRisebro-book}, which gives continuity in time if for all $t_1,t_2\ge 0$ and $\phi\in C^\infty_0(\Br)$, where $\Br=[-r,r]\cap \Omega$, 
\begin{equation}\label{eq:kruzInt}
  \Big\lvert \int_\Br ((\zcon(x,t_2)-\zcon(x,t_1))\phi(x)\,dx \Big\rvert \le C_r  \norm{\phi'}_{L^\infty(\Br)}(\abs{t_2-t_1}+ \mathcal{O}(\Delta t)),
\end{equation}  
in addition to the $L^\infty$ and $BV$ bound from Lemma~\ref{lem:bounds}.
\begin{lemma}\label{lem:ztime}
Let $\zcon$ be the piecewise constant interpolation of $z^n_j=\Dp w^n_j$, where $w^n_j$ is the solution of~\eqref{eq:wScheme}. Then $\zcon$ satisfies for any $t,t+\tau\ge 0$, $r>0$, 
\begin{equation*}
  \int_\Br \abs{\zcon(x,t+\tau)-\zcon(x,t)}\le C_r \max(\normBV{z^0},1)(\sqrt{\abs{\tau}}+\frac{\Delta t}{\sqrt{\abs{\tau}}})\,,
\end{equation*}  
where $\Br=[-r,r]\cap \Omega$.
\end{lemma}
\begin{proof}
  To apply Kru\v{z}kov's interpolation lemma, we need to show~\eqref{eq:kruzInt}. First, note that for any time step $n$,
\begin{align*}
  \Big\lvert \int_\Omega (\zcon(x,t^{n+1})-\zcon(x,t^n))\phi\,dx\Big\rvert&=\Big\lvert \sum_j (z^{n+1}_j-z^n_j)\int_{x_j}^{x_{j+1}}\phi\,dx\Big\rvert\\
&= \Big\lvert \sum_j \Delta t \Dp(B(w^{n+\theta}_j) \Dm z^{n+\theta}_j)\int_{x_j}^{x_{j+1}}\phi\,dx\Big\rvert\\
&= \Big\lvert \sum_j \Delta t B(w^{n+\theta}_j) \Dm z^{n+\theta}_j \frac{1}{\Delta x} \int_{x_j}^{x_{j+1}}\phi(x)-\phi(x-\Delta x)\,dx\Big\rvert\\
&\le \sum_j \Delta t \Delta x \abs{\Dm z^{n+\theta}_j} \normLinf{\phi'}\\
&\le \Delta t \normLinf{\phi'} \normBV{z^0}\,.
\end{align*}
For given $t_1,t_2>0$, let $n,m$ be such that $t_1\in [t^n,t^{n+1})$ and $t_2\in [t^m,t^{m+1})$. The above estimate yields
\begin{align*}
  \Big\lvert \int_\Omega (\zcon(x,t_2)- \zcon(x,t_1)) \phi \,dx\Big\rvert &\le   \normLinf{\phi'}\normBV{z^0}(t^m-t^n)\\
&\le  \normLinf{\phi'}\normBV{z^0}(t_2-t_1 + 2\Delta t)\,.
\end{align*}
Kru\v{z}kov's interpolation lemma~\cite[p.~208, Lemma~4.11]{HoldenRisebro-book} then implies
\begin{equation*}
  \int_\Br \abs{\zcon(x,t+\tau)-\zcon{x,t}}\,dx \le C_r (\epsilon + \epsilon \normBV{z^0} + \normBV{z^0} \frac{\abs{\tau}+2\Delta t}{\epsilon} ),
\end{equation*}
for any $\epsilon>0$. Choosing $\epsilon=\sqrt{\abs{\tau}}$, we arrive at the claim.
\end{proof}
\section{Convergence}\label{sec:convergence}
Finally, we are able to prove the convergence of the scheme, Theorem~\ref{thm:convergence}.
\begin{proof}[Proof of Theorem~\ref{thm:convergence}]
  We will apply Kolomogorov's compactness theorem~\cite[Thm.~A.11, p.~437]{HoldenRisebro-book} twice, first on $\wlin$ and then on $\zcon$, to get a subsequence of $\wlin$ that converges strongly in $C([0,\infty), W^{1,1}(\Omega))$.

For the compactness of $\wlin$, the $L^\infty$ and $BV$ bound on $w^n_j$ from Lemma~\ref{lem:bounds} imply, for $t\in[t^n,t^{n+1})$,
\begin{align*}
 & \normLinf{\wlin(t)}\le \normInf{w^n}\le \normInf{w^0}\le \normLinf{w_0}\le C\,,\\
&\normBV{\wlin(t)}=\normBV{w^n}\le \normBV{w^0}\le C \normLone{w_0'}\le C\,,
\end{align*}
where the constants on the right-hand side are independent of $\Delta t$.
Together with the time continuity from Lemma~\ref{lem:wtime}, Kolmogorov's theorem guarantees that a subsequence of $\wlin$ converges in $C([0,\infty), L^1(\Omega))$.

Similarly, for $\zcon$, we have from Lemma~\ref{lem:bounds},
\begin{align*}
 & \normLinf{\zcon(t)}=\normInf{z^n}\le \normInf{z^0}\le \normLinf{w'_0}\le C\,,\\
&\normBV{\zcon(t)}=\normBV{z^n}\le \normBV{z^0}\le C\normBV{w_0'}\le C\,.
\end{align*}
Because of the time continuity of $\zcon$ from Lemma~\ref{lem:ztime} and Kolomogorov's theorem, we can thus take another subsequence (for simplicity, we omit the subindices in the following) such that both $\wlin$ and $\zcon$ converge in $C([0,\infty), L^1(\Omega))$. Let $w$ and $z$ denote the corresponding limits.

For the piecewise constant interpolation $\wcon$, recall from~\eqref{eq:wconDiff} that for any $t\ge 0$,
\begin{equation*}
  \normLone{\wlin(t)-\wcon(t)}\le C \Delta x\,,
\end{equation*}
where $C$ is independent of $t$. Hence, also $\wcon$ converges to $w$ in $C([0,\infty), L^1(\Omega))$. 
Moreover, if we define
\begin{align*}
  \wthcon&=\theta\, \wcon(\cdot+\Delta t)+(1-\theta)\,\wcon,\\
\zthcon&=\theta\,\zcon(\cdot+\Delta t)+ (1-\theta)\,\zcon,
\end{align*}
then due to the time continuity from Lemma~\ref{lem:wtime} and~\ref{lem:ztime}, also $\wthcon$ and $\zthcon$ converge in $C([0,\infty),L^1(\Omega))$ to $w$ and $z$, respectively.

Because $\Dp w^n_j=z^n_j$, we have that for any $\phi_j=\phi(x_j)$, $\phi \in C^\infty_c(\Omega)$,
\begin{equation*}
  \sum_j w^n_j \Dm \phi_j = - \sum_j z^n_j \phi_j\,.
\end{equation*} 
Passing to the limit, we get
\begin{equation*}
  \int_\Omega w \px{\phi} \,dx=- \int_\Omega z \phi \, dx\,,
\end{equation*}
i.e., $\px{w}=z$.

Next, let $\phi$ be a test function in $C^\infty_c(\Omega\times [0,\infty))$ and set $\phi^n_j=\phi(x_j,t^n)$. Multiplying the equation of the scheme,~\eqref{eq:wScheme}, by $\phi^n_j$ and summing in $j$ and~$n$, we get
\begin{equation*}
  \sum_{n\ge 0}\sum_j \Dtp w^n_j \, \phi^n_j = \sum_{n\ge 0}\sum_j B(w^{n+\theta}_j) \DD w^{n+\theta}_j\, \phi^n_j\,,
\end{equation*}
which is the same as
\begin{align*}
  \sum_{n\ge 0 }\sum_j w^{n+1}_j \Dtp \phi^n_j + \sum_j w^0_j \phi^0_j &= \sum_{n\ge 0}\sum_j \Dp w^{n+\theta}_j \Dp \phi^n_j B(w^{n+\theta}_j)\\&\hphantom{\sum_{n\ge 0}\sum_j}\qquad{}+ \Dp w^{n+\theta}_j \Dp B(w^{n+\theta}_j) \phi^n_{j+1}\,,
\end{align*}
or
\begin{multline}\label{eq:conv1}
  \int_0^\infty \int_\Omega \wcon(t+\Delta t) \overline{\Dtp \phi^{\Delta t}}\,dxdt + \int_\Omega \wcon(x,0)\overline{\phi^{\Delta t}}(x,0)\,dx \\ =\int_0^\infty \int_\Omega \zthcon\, \overline{\Dp \phi^{\Delta t}} B(\wthcon) + \zthcon\, \overline{\Dp B(w^{\theta,\Delta t})}\,\overline{\phi^{\Delta t}}(x+\Delta x,t)\,dxdt\,,
\end{multline}
where $\overline{\Dtp \phi^{\Delta t}}$, $\overline{\phi^{\Delta t}}$, etc. denote the piecewise constant interpolations corresponding to $\Dtp \phi^n_j$, $\phi^n_j$, etc.

Since $\phi\in C^\infty_c([0,\infty)\times \Omega)$, we have that $\overline{\Dtp \phi^{\Delta t}}$, $\overline{\phi^{\Delta t}}$, $\overline{\Dp \phi^{\Delta t}}$ converge in $L^\infty(\Omega\times[0,\infty))$ to $\pt{\phi}$, $\phi$, and $\px{\phi}$, respectively. Furthermore, by the construction of the initial data,~\eqref{eq:ID}, 
\begin{equation*}
  \normLinf{\wcon(\cdot,0)-w_0}\le \Delta x \normLone{w'_0}\rightarrow 0\,,\quad \text{as $\Delta x\rightarrow 0$.}
\end{equation*}
It follows that the left-hand side of~\eqref{eq:conv1} converges to
\begin{equation*}\label{eq:convLHS}
  \int_0^\infty \int_\Omega w (x,t)\pt{\phi}(x,t) \,dxdt + \int_\Omega w_0 (x)\phi(x,0)\,dx.
\end{equation*}

For the right-hand side, since $B\in C^2(\R)$, the convergence of $\wthcon$ also implies the convergence of $B(\wthcon)$ to $B(w)$ and of $B'(\wthcon)$ to $B'(w)$. Furthermore,
\begin{equation*}
  \Dp B(w^{n+\theta}_j)= B'(w^{n+\theta}_j) \Dp w^{n+\theta}_j + \frac{\Delta x}{2} B''(\xi) (\Dp w^{n+\theta}_j)^2\,,
\end{equation*}
for some $\xi\in\R$, so
\begin{align*}
\normLone{\overline{\Dp B(w^{\theta,\Delta t})}-B'(w) \px{w}}&\le \normLone{B'(\wthcon)\zthcon - B'(w)\px{w}}\\&\qquad {}+ \frac{\Delta x}{2} \normLinfR{B''}\normLone{\zthcon}\normLinf{\zthcon}\\
&\le \normLone{B'(\wthcon)-B'(w)}\normLinf{\zthcon} \\&\qquad {}+ \normLinf{B'(w)}\normLone{\zthcon-\px{w}}\\ &\qquad+  \frac{\Delta x}{2} \normLinfR{B''}\normLone{\zthcon}\normLinf{\zthcon}\\
&\rightarrow 0,
\end{align*}
uniformly in $t$ as $\Delta x\rightarrow 0$.
Altogether, this allows us to pass to the limit also on the right-hand side of~\eqref{eq:conv1} to get
\begin{align*}
  \int_0^\infty \int_\Omega B(w) \px{w} \px{\phi} + B'(w) (\px{w})^2 \phi\,dxdt,
\end{align*}
which together with~\eqref{eq:convLHS} gives the weak formulation~\eqref{eq:weakform}.
\end{proof}
\section{Numerical experiments}\label{sec:numerics}
As mentioned in the introduction, weak solutions of~\eqref{eq:weq} are not necessarily unique, see also the analysis of Ughi et al.~\cite{Ughi1984,DalPassoLuckhaus1987,BertschDalPassoUghi1992} for the special case $B(w)=w$. The following experiments show the nonuniqueness for $B(w)=c^2(\bar{k}_w(w))$, where $c$ is given by~\eqref{eq:cLC} with $k_1=0$ and $k_2=1$, i.e., $c^2(u)=\sin^2(u)$. Then the transformation from $u\in [0,\pi]$ to $w$ is given by
\begin{equation*}
  w=k_w(u)=\int_{\frac{\pi}{2} }^{u} c(\xi)\, d\xi=-\cos(u)\,,
\end{equation*}
so 
\begin{equation*}
  B(w)=c^2(\bar{k}_w(w))=\sin^2(\arccos(-w))=1-w^2\,.
\end{equation*}

In the first series of experiments below we will construct the ``viscosity solution'' of Ughi et al. This is achieved by choosing grid points such that $\abs{w_0(x_j)}<1$, i.e., $B(w_0(x_j))\neq 0$ for all $j$. We will see that in this case the method converges and the limit is the same as the limit that one obtains by letting $k_1\rightarrow 0$ (for any set of grid points) or using a method for $v$ based on~\eqref{eq:v}.

Let the initial data be given by
\begin{equation*}
  u_0(x)= \begin{cases} -2\pi x + \frac{\pi}{2} \,, & \text{for $x\in [0,\frac{1}{4} ]$,}\\
    2\pi x- \frac{\pi}{2} \,, &\text{for $x\in [\frac{1}{4} ,\frac{3}{4} ]$,}\\ -2\pi x + \frac{5}{2} \pi\,,&\text{for $x \in [\frac{3}{4} ,1]$,} \end{cases}
\end{equation*}
i.e., 
\begin{equation*}
 w_0(x)=-\sin(2\pi x)\,, \quad \text{for $x\in [0,1]$,}
\end{equation*}
or 
\begin{equation*}
v_0(x)=\begin{cases} -\tan(2\pi x) \,, & \text{for $x\in [0,\frac{1}{4} ]$,}\\
    \tan(2\pi x)\,, &\text{for $x\in [\frac{1}{4} ,\frac{3}{4} ]$,}\\
-\tan(2\pi x)\,,&\text{for $x \in [\frac{3}{4} ,1]$,} \end{cases}  
\end{equation*}
where $v_0=\int_{\pi/2}^{u_0}\frac{1}{c(\xi)}\,d\xi$, see also Figure~\ref{fig:ID}. In all of the following experiments we will construct the discrete initial data directly by setting $w^0_j=w_0(x_j)$ ($v^0_j=v_0(x_j)$ for the $v$-based scheme), instead of using~\eqref{eq:ID}.
\begin{figure}
  \centering
\begin{subfigure}[b]{0.3\textwidth}\centering
  \includegraphics[width=\textwidth]{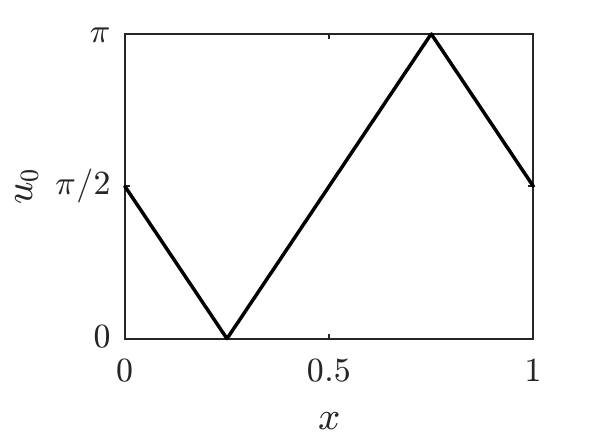}
\caption{$u_0$}
\end{subfigure}
\begin{subfigure}[b]{0.3\textwidth}\centering
  \includegraphics[width=\textwidth]{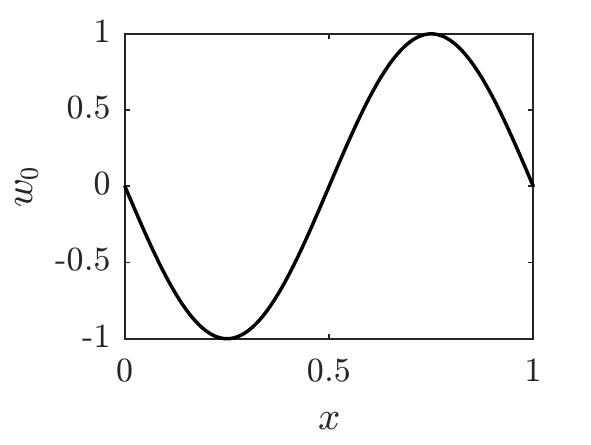}
\caption{$w_0$}
\end{subfigure}
\begin{subfigure}[b]{0.3\textwidth}\centering
  \includegraphics[width=\textwidth]{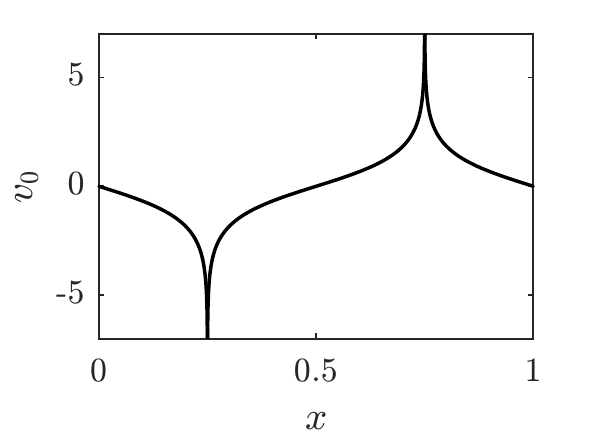}
\caption{$v_0$}
\end{subfigure}
  \caption{The initial data}
\label{fig:ID}
\end{figure}
 
Let $N$ be an odd number, so the grid points $x_j=j/N$ do not coincide with the critical points $1/4$ and $3/4$. For the time discretization, we choose to $\theta=1/2$ in~\eqref{eq:wScheme}, i.e., a Crank-Nicholson type discretization. The resulting implicit equation is solved using a standard Newton iteration. The time step is set to $\Delta t=100 (\Delta x)^2$. To check convergence, we calculate a solution $w^{\Delta t^*}$ on a fine grid ($N=100\cdot 2^8-1$) and define the errors
\begin{subequations}\label{eq:errs}
\begin{align}
  err_p=\norm{\overline{w^{\Delta t}}(\cdot,T)-\overline{w^{\Delta t^*}}(\cdot,T)}_{L^p(\Omega)}\,,\quad \text{$p\in\{1,\infty\}$}\,,\\
  err_{1,p}=\norm{w_x^{\Delta t}(\cdot,T)-w_x^{\Delta t^*}(\cdot,T)}_{L^p(\Omega)}\,,\quad \text{$p\in\{1,\infty\}$}\,,
\end{align}
\end{subequations}
where $T=0.04$. Table~\ref{tab:ConvWodd} shows that the numerical solutions with an odd number of grid points converge to $w^{\Delta t^*}$ with rate~$\approx 1$.
\begin{table}
\centering
\begin{tabular}{l | l l l l}
$N+1$ & $err_1$ & $err_{1,1}$&$err_\infty$ &$err_{1,\infty} $\\ \hline 
$100\cdot 2^{0}$ & \num{1.2e-01} & \num{1.5e+00}  & \num{3.9e-01}  & \num{4.1e+00}  \\ 
$100\cdot 2^{1}$ & \num{7.4e-03} (4.0) & \num{6.1e-02} (4.6) & \num{1.3e-02} (4.8) & \num{9.6e-02} (5.4) \\ 
$100\cdot 2^{2}$ & \num{1.5e-03} (2.3) & \num{1.6e-02} (2.0) & \num{2.7e-03} (2.3) & \num{2.9e-02} (1.7) \\ 
$100\cdot 2^{3}$ & \num{5.4e-04} (1.4) & \num{6.6e-03} (1.2) & \num{9.8e-04} (1.4) & \num{1.3e-02} (1.1) \\ 
$100\cdot 2^{4}$ & \num{2.5e-04} (1.1) & \num{3.2e-03} (1.1) & \num{4.5e-04} (1.1) & \num{6.6e-03} (1.0) \\ 
$100\cdot 2^{5}$ & \num{1.2e-04} (1.1) & \num{1.5e-03} (1.0) & \num{2.1e-04} (1.1) & \num{3.3e-03} (1.0) \\ 
\end{tabular}  \\[.5em]
\caption{$L^1$ and $L^\infty$ errors and rates (in brackets) of the numerical solutions and their derivatives at time $T=0.04$ for the scheme based on $w$ with $k_1=0$, $k_2=1$, $\theta=1/2$, CFL number $\Delta t/(\Delta x)^2=100$, and an odd number of grid points.}
\label{tab:ConvWodd}
\end{table}

Next, we calculate numerical solutions for $k_1=10^{-n}$, $n=1,\dots,5$. If $k_1$ and $k_2$ are positive, the transformation $k_w$ is given by 
\begin{equation*}
  k_w(u)=\int_{\pi/2}^u \sqrt{k_1 \cos^2(\xi)+k_2 \sin^2(\xi)} \,d\xi= k_2\, E\Big(u-\frac{\pi}{2} \,\Big|
\, 1- \frac{k_1}{k_2} \Big),
\end{equation*}
where $E(u\,|\,m)$ is the elliptic integral of the second kind. Because the function $B(w)=c^2(\bar{k}_w(w))$ does not have an explicit form, another Newton iteration is needed to solve for $\bar{k}_w$. 
In practice, this significantly slows down the method  and a scheme based on~\eqref{eq:VHE} or~\eqref{eq:v} would be preferable. Figure~\ref{fig:k1lim} shows that for a fixed number of grid points\footnote{ In Figure~\ref{fig:k1lim} we chose $N=400$, but for other $N$, in particular also for odd $N$, the result is the same.}, as $k_1\rightarrow 0$, the solutions converge to the same $w^{\Delta t^*}$ as above.
\begin{figure}
  \centering
  \begin{subfigure}[b]{.45\textwidth}\centering
  \includegraphics[width=\textwidth]{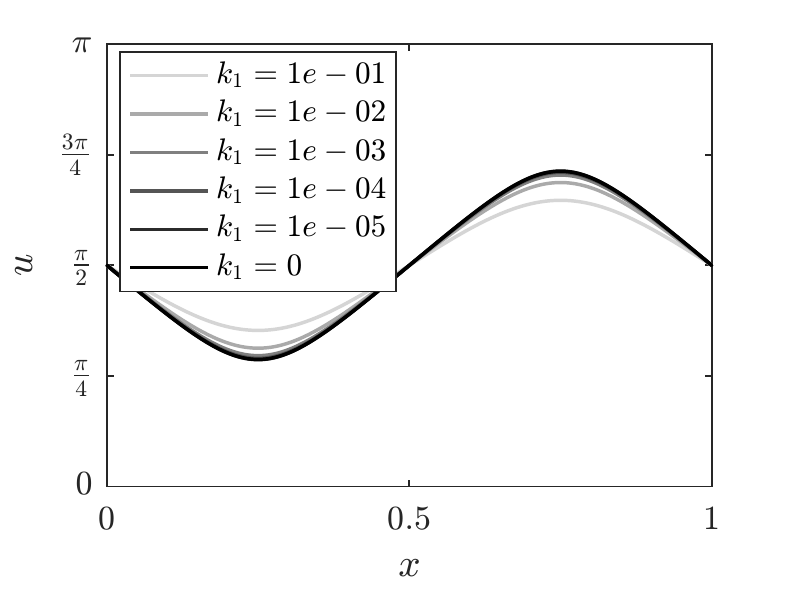}
\caption{$u$}    
  \end{subfigure}
  \begin{subfigure}[b]{.45\textwidth}\centering
  \includegraphics[width=\textwidth]{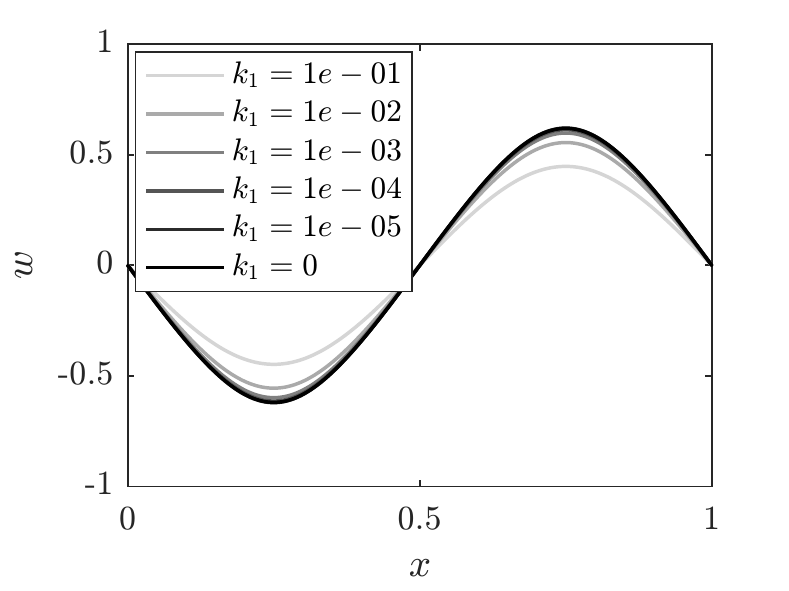}
\caption{$w$}    
  \end{subfigure}
  \caption{Convergence of solutions to viscosity solution as $k_1\rightarrow 0$. The plots show the solutions at $T=0.04$ for the scheme based on $w$ with $k_2=1$, $\theta=1/2$, CFL number $\Delta t/(\Delta x)^2=100$, and $N=400$ for $k_1>0$ and $N=399$ for $k_1=0$.}\label{fig:k1lim}
\end{figure}

Another way to obtain the viscosity solution is to use the transformation to $v$ variables,~\eqref{eq:v}. A straightforward scheme based on~\eqref{eq:v} is
\begin{equation}\label{eq:vScheme}
  \Dtp v^n_j=\Dp\big(A_-c^2(\bar{k}_v(v))^n_j\, \Dm v^n_j\big),,
\end{equation}
where $A_-c^2(\bar{k}_v(v))^n_j=\frac{1}{2}(c^2(\bar{k}_v(v^n_j))+c^2(\bar{k}_v(v^n_{j-1})))$.
For $c$ given by~\eqref{eq:cLC}, we have
\begin{equation*}
  k_v(u)=\int_{\pi/2}^u \frac{1}{\sqrt{k_1\cos^2(\xi)+k_2\sin^2(\xi)}}\,d\xi=\frac{1}{k_2} F\Big(u-\frac{\pi}{2}\, \Big|\, 1- \frac{k_1}{k_2} \Big)\,, 
\end{equation*}
where $F(u\,|\,m)$ is the elliptic integral of the first kind. Using Jacobi's amplitude function ``$\mathrm{am}$'', the inverse $\bar{k}_v$ can be expressed as
\begin{equation*}
  \bar{k}_v(v)=\mathrm{am}\Big(k_2 v\,\Big|\,1-\frac{k_1}{k_2} \Big)+ \frac{\pi}{2} \,.
\end{equation*}
For $k_1=0$ this method is only applicable if none of the grid points is a zero of $c(u_0(x))$, because $v$ would not be finite at such a point. Table~\ref{tab:vConv} shows the convergence of the $v$-based method to $w^{\Delta t^*}$ for an odd number of grid points and $k_1=0$. The errors in Table~\ref{tab:vConv} are calculated with $w^{\Delta t}$ and $w^{\Delta t^*}$ in~\eqref{eq:errs} replaced by the $u^{\Delta t}$ (the linear interpolation of $\bar{k}_v(v^n_j)$) and $u^{\Delta t^*}$ (the linear interpolation of $\bar{k}_w(w^n_j)$), respectively.
\begin{table}
\centering
\begin{tabular}{l | l l l l}
$N + 1$ & $err_1$ & $err_{1,1}$&$err_\infty$ &$err_{1,\infty}$\\ \hline 
$100\cdot 2^{0}$ & \num{5.2e-03}  & \num{6.1e-02}  & \num{1.0e-02}  & \num{1.3e-01}  \\ 
$100\cdot 2^{1}$ & \num{2.6e-03} (1.0) & \num{3.0e-02} (1.0) & \num{5.0e-03} (1.0) & \num{6.6e-02} (1.0) \\ 
$100\cdot 2^{2}$ & \num{1.3e-03} (1.0) & \num{1.5e-02} (1.0) & \num{2.5e-03} (1.0) & \num{3.3e-02} (1.0) \\ 
$100\cdot 2^{3}$ & \num{6.1e-04} (1.0) & \num{7.3e-03} (1.0) & \num{1.2e-03} (1.0) & \num{1.7e-02} (1.0) \\ 
$100\cdot 2^{4}$ & \num{2.9e-04} (1.1) & \num{3.6e-03} (1.0) & \num{5.8e-04} (1.1) & \num{8.2e-03} (1.0) \\ 
$100\cdot 2^{5}$ & \num{1.4e-04} (1.1) & \num{1.7e-03} (1.1) & \num{2.7e-04} (1.1) & \num{3.9e-03} (1.1) \\ 
\end{tabular}\\[.5em]
\caption{$L^1$ and $L^\infty$ errors and rates (in brackets) of the numerical solutions and their derivatives compared to the ``viscosity solution''  at time $T=0.04$ for the scheme based on $v$,~\eqref{eq:vScheme}, with $k_1=0$, $k_2=1$, $\theta=1/2$, CFL number $\Delta t/(\Delta x)^2=100$, and an odd number of grid points.}
\label{tab:vConv}
\end{table}

Finally, we construct a weak solution of the $w$-equation different from the viscosity solution $w^{\Delta t^*}$ by choosing an even number of grid points in the scheme defined by~\eqref{eq:wScheme}. By definition, if $B(w^0_j)=0$, we have $B(w^n_j)=0$ for all $n$. This differs from the solution above, where at $T=0.04$ we have $B(w^{\Delta t^*}(x,T))>0$ at all $x$. Figures~\ref{fig:Evol}--\ref{fig:Evol2} show the evolution of the two solutions in time. The errors in Table~\ref{tab:ConvWeven} are calculated as in~\eqref{eq:errs}, with $w^{\Delta t^*}$ replaced by the numerical solution for $N=100\cdot 2^8$ grid points. The results confirm the convergence of~\eqref{eq:wScheme} for an even number of grid points. The decreasing convergence rates for the derivatives are due to the fact that the error is calculated using an approximation of the exact solution. 
Intuitively, the second solution corresponds to solutions of several Dirichlet boundary value problems with the boundary points given by the points where $B(w^0_j)=0$.  As $T\rightarrow \infty$, the function $w(x,t)$ converges to
\begin{equation*}
  w_\infty(x)= \begin{cases} -4 x\,, & \text{for $x\in [0,\frac{1}{4} ]$,}\\
    4 x- 2 \,, &\text{for $x\in [\frac{1}{4} ,\frac{3}{4} ]$,}\\ 
4 - 4  x\,,&\text{for $x \in [\frac{3}{4} ,1]$,} \end{cases}
\end{equation*}
which means that $u(x,t)$ tends to
\begin{equation*}
  u_\infty(x)= \begin{cases} \arccos(4x)\,, & \text{for $x\in [0,\frac{1}{4} ]$,}\\
   \arccos(2-4x) \,, &\text{for $x\in [\frac{1}{4} ,\frac{3}{4} ]$,}\\ 
\arccos(4x-4)\,,&\text{for $x \in [\frac{3}{4} ,1]$,} \end{cases}  
\end{equation*}
 and thus $u_{x}(x,t)\rightarrow \infty$ at $x=\frac{1}{4}$ and $\frac{3}{4}$ as $t\rightarrow \infty$.
\begin{figure}
  \centering
  \begin{subfigure}[b]{.45\textwidth}\centering
  \includegraphics[width=\textwidth]{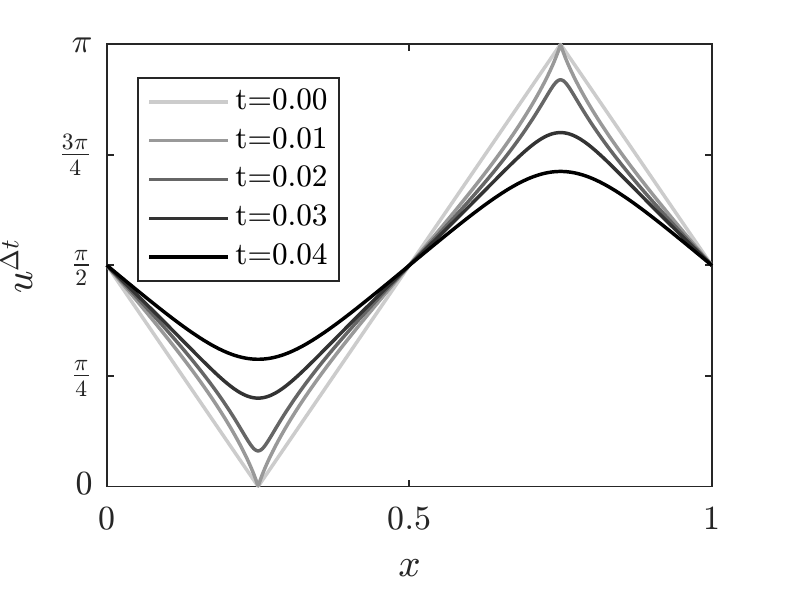}
\caption{$u$}    
  \end{subfigure}
  \begin{subfigure}[b]{.45\textwidth}\centering
  \includegraphics[width=\textwidth]{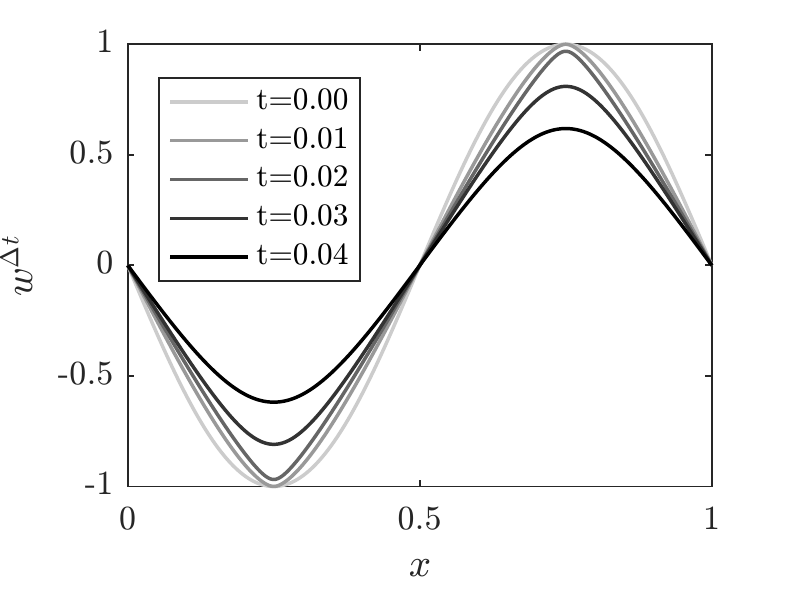}
\caption{$w$}    
  \end{subfigure}
  \caption{Time evolution of the viscosity solution $w^{\Delta t^*}$ (limit when $N$ is odd)}\label{fig:Evol}
\end{figure}
\begin{figure}
  \centering
  \begin{subfigure}[b]{.45\textwidth}\centering
  \includegraphics[width=\textwidth]{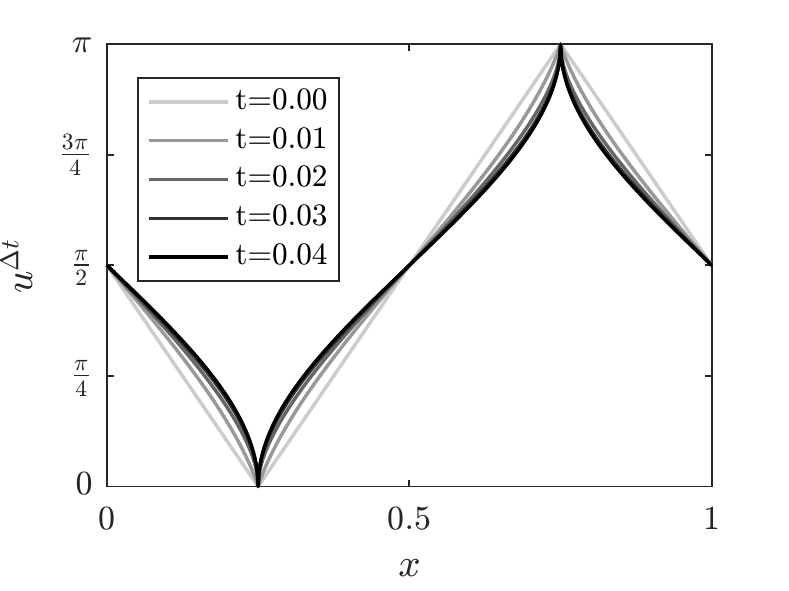}
\caption{$u$}    
  \end{subfigure}
  \begin{subfigure}[b]{.45\textwidth}\centering
  \includegraphics[width=\textwidth]{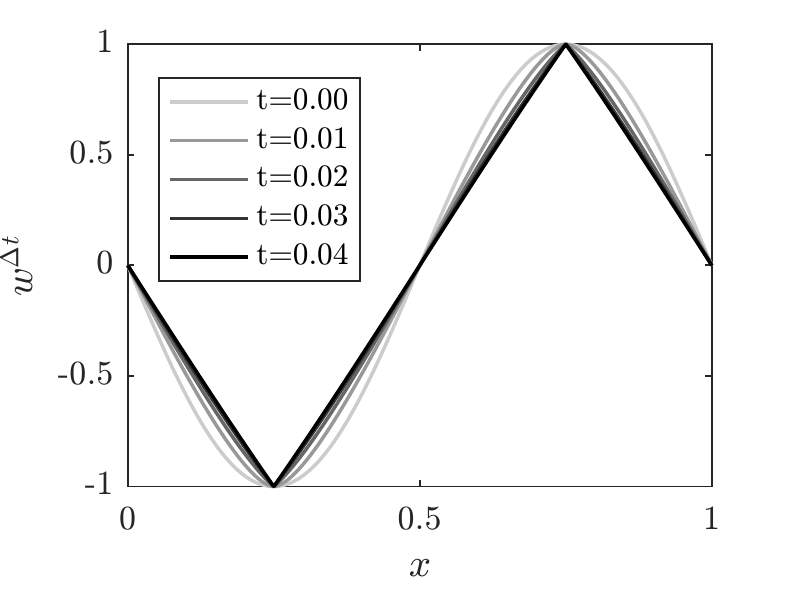}
\caption{$w$}    
  \end{subfigure}
  \caption{Time evolution of the second weak solution (limit when $N$ is even)}\label{fig:Evol2}
\end{figure}
\begin{table}
\centering
\begin{tabular}{l | l l l l}
$N$ & $err_1$ & $err_{1,1} $&$err_\infty$ &$err_{1,\infty} $\\ \hline 
$100\cdot 2^{0}$ & \num{1.3e-03}  & \num{1.9e-02}  & \num{2.3e-03}  & \num{5.1e-02}  \\ 
$100\cdot 2^{1}$ & \num{1.1e-04} (3.6) & \num{2.3e-03} (3.1) & \num{1.7e-04} (3.8) & \num{7.4e-03} (2.8) \\ 
$100\cdot 2^{2}$ & \num{1.3e-05} (3.0) & \num{9.3e-04} (1.3) & \num{2.0e-05} (3.0) & \num{2.1e-03} (1.8) \\ 
$100\cdot 2^{3}$ & \num{2.7e-06} (2.3) & \num{4.7e-04} (1.0) & \num{4.1e-06} (2.3) & \num{1.0e-03} (1.1) \\ 
$100\cdot 2^{4}$ & \num{6.9e-07} (2.0) & \num{2.7e-04} (0.8) & \num{1.1e-06} (2.0) & \num{6.0e-04} (0.7) \\ 
$100\cdot 2^{5}$ & \num{1.9e-07} (1.9) & \num{1.9e-04} (0.5) & \num{2.9e-07} (1.9) & \num{3.8e-04} (0.7) \\ 
\end{tabular}  \\[.5em]
\caption{$L^1$ and $L^\infty$ errors and rates (in brackets) of the numerical solutions and their derivatives at time $T=0.04$ for the scheme based on $w$ with $k_1=0$, $k_2=1$, $\theta=1/2$, CFL number $\Delta t/(\Delta x)^2=100$, and an even number of grid points.}
\label{tab:ConvWeven}
\end{table}


\end{document}